\documentclass[11pt]{amsart}

\newtheorem*{MT}{Main Theorem}
\newtheorem{lemma}{Lemma}[section]
\newtheorem{prop}[lemma]{Proposition}
\theoremstyle{definition}

\theoremstyle{remark}

\numberwithin{equation}{section}
\def\R{{\mathbb R}}
\def\Z{{\mathbb Z^n}}
\newcommand{\eps}{\epsilon_\star}

\usepackage{amssymb,amsgen,amsbsy,amsopn,graphicx}
\usepackage{amsfonts,amsmath}
\begin{document}

\title[Compressible Euler with damping on Torus]{Compressible Euler equation with damping on Torus in arbitrary dimensions}
\author{{\bf Nan Lu}\\\textsl{Lehigh University}}

\address{14 East Packer Avenue\\Department of Mathematics\\Lehigh University\\
Bethlehem, PA, 18015}

\email{nal314@lehigh.edu}
\maketitle
\begin{abstract}
We study the exponential stability of constant steady state of isentropic compressible Euler equation with damping on $\mathbb T^n$. The local existence of solutions is based on semigroup theory and some commutator estimates. We propose a new method instead of energy estimates to study the stability, which works equally well for any spatial dimensions.
\end{abstract}
\section{Introduction}\label{Intro}
We consider the compressible Euler equation with frictional damping in a periodic box $[0,1]^n$, which takes the form
\begin{equation}\label{eq1}
\left\{\begin{aligned}
&\rho_t+\nabla\cdot(\rho U)=0,\\
&(\rho U)_t+\nabla\cdot(\rho U\otimes U)+\nabla P=-\alpha\rho U,\\
&(\rho,U)(x,0)=(\rho_0,U_0)(x),\\
&(\rho,U)(t,\cdots,x_i,\cdots)=(\rho,U)(t,\cdots,x_i+1,\cdots)\in R^{n+1}\ , \ t\geq0,\\
&\int_{[0,1]^n}\rho_0\ dx=\overline\rho>0.
\end{aligned}\right.
\end{equation}
Our goal is show the exponential stability of the steady state $(\overline\rho,0)$.

Such a system occurs in the mathematical modeling of compressible flow through a porous
medium. Here the unknowns $\rho, U$ and $P$ denote the density, velocity and pressure, respectively.
The constant $\alpha>0$ models friction. We assume the flow is a polytropic perfect gas, then $P(\rho)=P_0\rho^\gamma$, where $\gamma>1$ is the adiabatic gas exponent. To keep our exposition clean (but not loss of generality), we will take $P_0=\frac1\gamma,\alpha=1$ and $\overline\rho=1$ starting from section~\ref{SU}.

There exist extensive literatures in past decades about compressible Euler equation subject to various initial and initial-boundary conditions. Both classical and weak solutions are constructed and their long time behavior are investigated. For the Cauchy problem, we refer the readers to \cite{HL92,STB03,WY01,WY07} and references therein for the existence of small smooth solutions; to \cite{CDL89,HP06,TW12} for $L^\infty$ solutions.
For large time behavior of solutions, we refer to \cite{N96,NWY00,Z01} for small smooth solutions and \cite{HMP05,Zh03} and references therein for weak solutions. In the direction of initial-boundary value problems, we refer to \cite{HP00,MM00,NY99} for small solutions and \cite{PZ08} for $L^\infty$ solutions.
For non-isentropic flows, see \cite{P01,P06} and references therein. It is also worth to mention some recent work \cite{BHN07,BZ11,MN10} for hyperbolic systems with partial dissipations in several spatial dimensions.

With the notation introduced in section~\ref{SU}, our main result is the following:
\begin{MT}
Let $s>\frac n2$. There exists $\eps>0$ and $\delta>0$ such that if $\|\rho_0-\overline\rho\|_{s+1}+\|U_0\|_{s+1}<\eps$, there exists a unique solution with initial value $(\rho_0,U_0)$ such that
\[(\rho,U)\in C^0([0,\infty),X_{s+1})\cap C^1([0,\infty),X_{s}),\] which also satisfies
\[\|\rho(t,\cdot)-\overline\rho\|_{s+1}+\|U(t,\cdot)\|_{s+1}\leq C(\|\rho_0-\overline\rho\|_{s+1}+\|U_0\|_{s+1})e^{-\delta t}.\]
\end{MT}
The result in this paper is not new (especially in low dimensions) for researchers working on compressible Euler equations. However, the method we use is new in the literature. Moreover, the proof is simpler and shorter compared to earlier works and applies to any spatial dimensions equally well.

Energy estimates is the standard approach for analyzing global existence and asymptotic behavior of partial differential equations. However, such method gets a little tediously long when the spatial dimension increases, because the function spaces needs to more regular in order to make the equation well-posed. The main contribution of this paper is contained in subsection \ref{AB}, where we propose a new method to obtain the decay property of small solutions. Our strategy is to make a change of coordinate so that the solution lives on a subspace under such coordinate system. Unlike the usual coordinate system using a fixed frame, we define ours by using a moving frame which depends on the solution itself. Under such moving frame, we can decompose the solution into two parts and discover the decay property.

The rest of this paper is organized as follows. In Section~\ref{SU}, we introduce some notations and rewrite \eqref{eq1} in a dynamical system form. In Section~\ref{proof}, we present the proof of the Main Theorem.

\section{Set Up}\label{SU}
Throughout this paper, we use $H^s$ to denote Sobolev space of periodic functions equipped with norm $\|\cdot\|_s$, i.e.,
\[\|f\|_s\triangleq(\sum_{|\alpha|\leq s}\|D^{\alpha}f\|^2_{L^2})^{\frac 12},\]
where $\alpha$ is a multi-index.

For any vector valued function $F=(f_1,\cdots,f_m):[0,1]^n\longrightarrow\R^m$,
\[\|F\|_s\triangleq\sum_{i=1}^n \|f_i\|_s.\]
For any $s\in \mathbb Z$, we define Banach spaces
\[\begin{aligned}
& X_s\triangleq\{(f,g)\big|f\in H^s,g\in\oplus^n H^s\}\ ,\ \tilde X_s\triangleq\{(f,g)\in X_s\big|\int_{[0,1]^n} f(x)\ dx=0\},\\
%& Z^{\pm}_{s,\delta}\triangleq\{(f,g)\big|(f,g)\in C^0([0,\infty),X_{s}), \sup_{t\geq0(\leq 0)}e^{\delta t}(\|f(t,\cdot)\|_{s}+\|g(t,\cdot)\|_{s})<\infty)\},
\end{aligned}\]
which are equipped with norms
\[\|(f,g)\|_s\triangleq\|f\|_s+\|g\|_s.\]%\ , \ |(f,g)|_{s,\delta}\triangleq \sup_{t\geq0}e^{\delta t}\big(\|f(t,\cdot)\|_{s}+\|g(t,\cdot)\|_{s}\big).\]

From now on, we will assume $P_0=\frac1\gamma,\alpha=1$ and $\overline\rho=1$. Let $\theta=\frac{\gamma-1}{2}$ and $\sigma=\frac{\rho^\theta-1}{\theta}$. Then the system \eqref{eq1} can be written as
\begin{equation}\label{eq2}
\left\{\begin{aligned}
&\begin{pmatrix}
\sigma\\
U\end{pmatrix}_t=\begin{pmatrix}
0 & -\nabla\cdot\\
-\nabla & -1\end{pmatrix} \begin{pmatrix}
\sigma\\U
\end{pmatrix}-\begin{pmatrix}
U\cdot\nabla & \theta\sigma\nabla\cdot\\
\theta\sigma\nabla & U\cdot\nabla\end{pmatrix}\begin{pmatrix}
\sigma\\U
\end{pmatrix},\\
&(\sigma,U)(x,0)=(\sigma_0,U_0)(x),\\
%&(\rho,U)(t,\cdots,x_i,\cdots)=(\rho,U)(t,\cdots,x_i+1,\cdots)\in R^{n+1}\ , \ t\geq0,\\
&\int_{[0,1]^n}\sigma_0\ dx=\overline\sigma.
\end{aligned}\right.
\end{equation}
It is clear that classical solutions of \eqref{eq1} and \eqref{eq2} are equivalent through the transformation $\sigma=\frac{\rho^\theta-1}{\theta}$. Let
\[A\triangleq\begin{pmatrix}
0 & -\nabla\cdot\\
-\nabla & -1\end{pmatrix} \ , \ B_{\sigma,U}\triangleq-\begin{pmatrix}
U\cdot\nabla & \theta\sigma\nabla\cdot\\
\theta\sigma\nabla & U\cdot\nabla\end{pmatrix}\]
and \eqref{eq2} can be written abstractly as
\begin{equation}\label{eq2.1}
(\sigma,U)_t=(A+B_{\sigma, U})(\sigma,U).
\end{equation}
%Let $\tilde \rho=\rho-1$ and $M=\rho U$, \eqref{eq1} can be written as
%\begin{equation}\label{eq2}
%\left\{\begin{aligned}
%&\tilde \rho_t=-\nabla \cdot M,\\
%&M_t=-(1+\tilde\rho)^{\gamma-1}\nabla\tilde \rho- M-\nabla\cdot(\frac{1}{1+\tilde \rho} MM^T ),\\
%&(\tilde \rho,M)(x,0)=(\tilde \rho_0,M_0)(x),\\
%&(\tilde \rho,U)(t,\cdots,x_i,\cdots)=(\tilde \rho,U)(t,\cdots,x_i+1,\cdots)\ , \ t\geq0,\\
%&\int_{[0,1]^n}\tilde \rho_0\ dx=0,
%\end{aligned}\right.
%\end{equation}
%where $M^T$ denotes the transpose of $M$. For simplicity, we will drop $\tilde{}$ sign on $\rho$ from now on. Define
%\[A=\begin{pmatrix}
%0 &-\nabla\cdot\\
%-\nabla & -1
%\end{pmatrix}\ , \ A_{\rho,M}=\begin{pmatrix}
%0 & -\nabla\cdot\\
%(-(1+\rho)^{\gamma-1}+\frac{MM^T}{(1+\rho)^2})\nabla & -1-\frac{L_M}{1+\rho}
%\end{pmatrix},\]
%where
%\[L_M(M_1)=(\nabla\cdot M_1)M+(\nabla M_1^T)\cdot M.\]
%Then $\eqref{eq2}_1$ and $\eqref{eq2}_2$ can be written in an abstract form as
%\begin{equation}\label{eq3}
%\begin{pmatrix}
%\rho\\M
%\end{pmatrix}_t=A_{\rho,M}\begin{pmatrix}
%\rho\\M
%\end{pmatrix}.
%\end{equation}

For any bounded linear operator $T\in L(X,Y)$, we use $|T|_{L(X,Y)}$ to denote the operator norm. We will drop the subscript if the context causes no confusion. For any Banach space $Z$, we use $B_\delta(Z)$ to the denote the ball centered at the origin of radius $\delta$ in $Z$.

\section{Proof}\label{proof}
In this section, we give the proof of the Main Theorem. We split this section into three subsections. In subsection \ref{Li}, we show that $A$ generates a semigroup $T(t)$ and $A+B_{\sigma,U}$ generates an evolutionary operator $M_{\sigma,U}$ for suitably chosen $(\sigma,U)$, respectively. To make our presentation self-contained, we prove local well-posedness in subsection \ref{local}. In subsection \ref{AB}, we analyze the asymptotic behavior of local solutions.
%\begin{remark}
%Due to the quasi-linear nature of the problem, one cannot use Duhamel's principle after Lemma 1 to define the contraction mapping.
%\end{remark}
\subsection{Semigroup and Evolutionary Operator.} \label{Li}
In this subsection, we provide estimates on the semigroup and the evolutionary operator generated by $A$ and $A+B_{\sigma,U}$ on $X_s$. We begin with some properties of $A$.
\begin{lemma}\label{le1}
The linear operator $A$ generates a strongly continuous semigroup $T(t)$ on $X_s$ for every $s$. Moreover, there exists $K\geq1$ such that for $t\geq0$, 
\begin{equation}\label{eq3.1}
|T(t)|_{L(X_s,X_s)}\leq K\ , \ |T(t)|_{L(\tilde X_s,\tilde X_s)}\leq Ke^{-\frac 12 t}.
\end{equation}
\end{lemma}
\begin{proof}
Since $A$ is a closed, densely defined linear operator and is the sum of an anti-selfadjoint and a bounded linear operator, the standard semigroup theory implies $A$ generates a $C_0$-semigroup $T(t)$ on $X_s$. Let $\hat A(\xi)$ be the symbol of $A$, namely,
\[\hat A(\xi)=\begin{pmatrix}
0 & -i\xi^T\\
-i\xi & -I
\end{pmatrix}\ , \ \xi\in\mathbb Z^n.\]
The above matrix has eigenvalues $-1$ with multiplicity $n-1$ and $\lambda_1=-\frac12+\frac12\sqrt{1-4|\xi|^2}$,
$\lambda_2=-\frac12-\frac12\sqrt{1-4|\xi|^2}$. For $\xi\in\mathbb Z^n/\{0\}$, we choose an orthonormal basis as follows
\[\begin{aligned}
&V_1=\frac{1}{\sqrt{|\xi|^2+|\lambda_2|^2}}\begin{pmatrix}
-i\lambda_2\\
\xi\end{pmatrix}\ , \ V_2=\begin{pmatrix}
h_0 \\h_1\end{pmatrix}=\begin{pmatrix}
\frac{i|\xi|^2}{\overline{\lambda_2}\sqrt{\frac{|\xi|^4}{|\lambda_2|^2}+|\xi|^2}}\\
\frac{\xi}{\sqrt{\frac{|\xi|^4}{|\lambda_2|^2}+|\xi|^2}}
\end{pmatrix},\\
&V_i=\begin{pmatrix}
0\\
\eta_i\end{pmatrix}\ , \ \eta_i\in\{\mathbb C\xi\}^\perp\ , \ 1\leq i\leq n-1.
\end{aligned}\]
Let $R(\xi)=(V_1,\cdots,V_{n+1})$ and we have $R(\xi)$ is unitary, namely, $R^{-1}(\xi)=R^{\star}(\xi)=\overline{R^t(\xi)}$. Consequently,
\[\begin{aligned}
&R^{-1}(\xi)\hat A(\xi)R(\xi)\\
=&\begin{pmatrix}
\overline{h_0} & \overline{h_1^t}\\
\frac{i\overline{\lambda_2}}{\sqrt{|\xi|^2+|\lambda_2|^2}} & \frac{\xi^t}{\sqrt{|\xi|^2+|\lambda_2|^2}}\\
0 & \eta_1^t\\
\vdots & \vdots\\
0 & \eta_{n-1}
\end{pmatrix}\begin{pmatrix}
0 & -i\xi^t\\
-i\xi & -1
\end{pmatrix}\begin{pmatrix}
h_0 & -\frac{i\lambda_2}{\sqrt{|\xi|^2+|\lambda_2|^2}}& 0 &\cdots & 0\\
h_1 & \frac{\xi}{\sqrt{|\xi|^2+|\lambda_2|^2}} & \eta_1 & \cdots & \eta_{n-1}
\end{pmatrix}\\
=&\begin{pmatrix}
\begin{matrix}
-i\overline{h_1^t}\cdot\xi h_0-i\overline{h_0}\xi^t\cdot h_1-1|h_1|^2 & \frac{-\overline{h_1^t}\cdot\xi\lambda_2-i\overline{h_0}|\xi|^2-1\overline{h_1^t}\cdot\xi}{\sqrt{|\xi|^2+|\lambda_2|^2}}\\
\frac{-i|\xi|^2h_0+(\overline{\lambda_2}-1)\xi^t\cdot h_1}{\sqrt{|\xi|^2+|\lambda_2|^2}} & \frac{(\overline{\lambda_2}-\lambda_2-1)}{|\xi|^2+|\lambda_2|^2}
\end{matrix} & 0\\
0 & - I_{n-1}
\end{pmatrix}\\
=&\begin{pmatrix}
\begin{matrix}
\lambda_2 & 0\\
-1 & \lambda_1
\end{matrix} & 0\\
0 & -I_{n-1}
\end{pmatrix}\triangleq B(\xi).
\end{aligned}\]
It follows that
\[e^{tB(\xi)}=\begin{pmatrix}
\begin{matrix}
e^{\lambda_2t} & 0\\
-\frac{e^{\lambda_2 t}-e^{\lambda_1 t}}{\lambda_2-\lambda_1} & e^{\lambda_1 t}
\end{matrix} & \\
 & e^{-t}I_{n-1}
\end{pmatrix}.\]
We note that for $\xi\in\mathbb Z^n/\{0\}$, there exists $K\geq1$ independent of $\xi$ such that
\[|\frac{e^{\lambda_2 t}-e^{\lambda_1 t}}{\lambda_2-\lambda_1}|\leq Ke^{-\frac 12t},\]
%which follows , 
%\begin{equation}\label{eq4}
%|e^{tB(\xi)}|_{L^(\mathbb C^{n+1},\mathbb C^{n+1})}\leq K_2e^{-\frac12 t}\ , \ t\geq0.
%\end{equation}
%The above bound also holds for $\xi=0$, this is because
%\[\hat\rho(0)=\int_{[0,1]^n}\rho\ dx=\int_{[0,1]^n}\rho_0\ dx=0.\]
%Therefore, there exists $K_3\geq1$ uniform in $\xi$ such that
which implies
\begin{equation*}\label{eq5}
|e^{t\hat A(\xi)}|_{L^(\mathbb C^{n+1},\mathbb C^{n+1})}=|e^{tR^{-1}(\xi)B(\xi)R(\xi)}|=|R^{-1}(\xi)e^{tB(\xi)}R(\xi)|\leq Ke^{-\frac 12t}.
\end{equation*}
For $\xi=0$, we have $\hat A(0)=\begin{pmatrix}0 & 0\\ 0 & -1\end{pmatrix}$. Therefore, for $t\geq0$,
\begin{equation}\label{eq6}
|T(t)|_{L(X_s,X_s)}\leq K\ , \ |T(t)|_{L(\tilde X_s,\tilde X_s)}\leq Ke^{-\frac 12 t}.
\end{equation}
\end{proof}
%The above lemma implies the resolvent of $A$ satisfies
%\begin{equation}\label{eq7}
%|(\lambda-A)^{-1}|_{L(X_s,X_s)}\leq\frac{K}{\lambda+\frac 12}\ \mbox{for every}\ \lambda>-\frac12.
%\end{equation}

\begin{prop}\label{prop2}
Let $s>\frac n2$ and $\eps>0$ be a sufficiently small constant. For any $(\sigma,U)$ such that $\|(\sigma,U)\|_{C^0([0,T],X_{s+1})}<\eps$, $A+B_{\sigma,U}$ generates an evolutionary system $M_{\sigma,U}(t,\tau)$ on $X_{s}$, which satisfies
\begin{equation}\label{eq7.1}
|M_{\sigma,U}(t,\tau)|_{L(X_{s},X_{s})}\leq e^{(C\eps+1)(t-\tau)}\ \ \text{for}\ \ 0\leq\tau\leq t\leq T,
\end{equation}
where $C$ only depends on $s$ and $n$.
\end{prop}
\begin{proof}
For any $(\sigma_1,U_1)\in X_{s+1}$, we have
\[\begin{aligned}
&<B_{\sigma,U}\begin{pmatrix}\sigma_1\\U_1\end{pmatrix},\begin{pmatrix}\sigma_1\\U_1\end{pmatrix}>_{H^s}\\
=&-\sum_{|\alpha|\leq s}\int_{[0,1]^n}D^\alpha\begin{pmatrix}U\cdot\nabla\sigma_1+\theta\sigma\nabla\cdot U_1\\
\theta\sigma\nabla\sigma_1+U\cdot\nabla U_1\end{pmatrix}\cdot D^\alpha\begin{pmatrix}\sigma_1\\U_1\end{pmatrix}\ dx\\
=&-\int_{[0,1]^n}(U\cdot D^s\nabla\sigma_1+\theta\sigma D^s\nabla\cdot U_1)D^s\sigma_1+(\theta\sigma D^s\nabla\sigma_1+U\cdot D^s\nabla U_1)\cdot D^s U_1\ dx+l.o.t\\
=&\int_{[0,1]^n}\nabla\cdot U(\frac 12(D^s\sigma_1)^2)+\nabla(\theta\sigma)\cdot(D^s\sigma_1D^s U_1)+\nabla\cdot U(\frac 12 U_1\cdot U_1)\ dx+l.o.t\\
\leq&C\|D(\sigma,U)\|_{L^\infty}\|(\sigma_1,U_1)\|_s^2\leq C\eps\|(\sigma_1,U_1)\|_s^2,
\end{aligned}\]
where $C$ only depends on $s$ and $n$. This means $B_{\rho,U}-C\eps$ is dissipative. 

%Next we introduce the following norm on $X_s$, 
%\[|x|_s\triangleq \sup_{t\geq0} \|T(t)x\|_s,\]
%where $T(t)$ is defined in Lemma \ref{le1}. It is clear that for all $x\in X_s$
%\begin{equation}\label{eq9}
%\|x\|_{s}\leq|x|_s\leq K\|x\|_s\ , \ |T(t)x|_s\leq |x|_s.
%\end{equation}
%Therefore, the new norm $|\cdot|_s$ is equivalent to $\|\cdot\|_s$ on $X_s$ and $A$ generates a contraction on $X_s$ under $|\cdot|_s$. 

Let $\tilde A=\begin{pmatrix}
0 & -\nabla\cdot\\
-\nabla & 0\end{pmatrix}$, i.e., $\tilde A=A+\begin{pmatrix}
0 & 0\\
0& I_{n\times n}\end{pmatrix}$. Since $\tilde A$ is anti-selfadjoint, it generates a unitary group $U(t)$ for $t\in\mathbb R$. Note that
\begin{equation*}\label{eq8}
\begin{aligned}
&\|(B_{\rho,U}-C\eps)(
\sigma_1,U_1)\|_{s}\\
\leq & \|U\cdot\nabla\sigma_1+\theta\sigma\nabla\cdot U_1\|_s+\|\theta\sigma\nabla\sigma_1+U\cdot\nabla U_1\|_s+C\eps\|(\sigma_1,U_1)\|_s\\
\leq & C(\|(\sigma,U)\|_{L^\infty}+\|D(\sigma,U)\|_{L^\infty})(\|\nabla\sigma_1\|_s+\|\nabla\cdot U_1\|_s)+C\eps\|(\sigma_1,U_1)\|_s\\
%\leq & C\eps(\|\nabla\sigma_1+U_1\|_s+\|\nabla\cdot U_1\|_s+\|(\sigma_1,U_1)\|_s)\\
\leq & C\eps\|A(\sigma_1,U_1)\|_s+C\eps\|(\sigma_1,U_1)\|_s.
\end{aligned}
\end{equation*}
%From \eqref{eq8} and \eqref{eq9}, we have
%\[|(B_{\rho,U}-C\eps)(
%\sigma_1,U_1)|_s\leq CK\eps(|A(\sigma_1,U_1)|_s+|(\sigma_1,U_1)|_s).\]
%Moreover, $B_{\sigma,U}-C\eps$ is still dissipative under the inner product induced by the norm $|\cdot|_s$. 
Therefore, by Corollary 3.3.3 in \cite{P83}, $\tilde A+B_{\rho,U}-C\eps$ generates a family of contractions on $X_s$ for small $\eps$. Since $\begin{pmatrix}
0 & 0\\
0 & -I_{n\times n}\end{pmatrix}$ is bounded with norm equal to $1$. Consequently, $A+B_{\rho,U}-C\eps$ generates an evolutionary operator $M_{\sigma,U}$ with satisfies \eqref{eq7.1}.
\end{proof}

\subsection{{\bf Local existence}.}\label{local} 
In this subsection, we prove the local existence of small initial data though a contraction mapping argument. Recall that we obtain the evolutionary system $M_{\sigma,U}(t,\tau)\in L(X_s,X_s)$ for small $(\sigma,U)\in X_{s+1}$. Therefore, we need to show such operator also maps $X_{s+1}$ to $X_{s+1}$. This fact can be proved by the following commutator estimate. The main idea is developed by Kato in \cite{K75}, which has wide applications in studying local well-posedness of quasi-linear equations. Let
\[Sf(x)\triangleq \sum_{\xi\in\Z}e^{2\pi i x\cdot\xi}(1+|\xi|^2)^\frac 12\hat f(\xi),\]
which is an isomorphism between $X_{s+1}$ and $X_s$ such that $|S|=|S^{-1}|=1$. 
\begin{lemma}\label{le2}
Let $s>\frac n2$ and $(\sigma, U)\in X_{s+1}$, then
\[S(A+B_{\sigma,U})S^{-1}-(A+B_{\sigma, U})\in L(X_s,X_s).\]
\end{lemma}
\begin{proof}
It is clear that $A$ commutes with $S$. Thus, we only need to show $SB_{\sigma, U}S^{-1}-B_{\sigma, U}\in L(X_s,X_s)$. The Fourier transform of the difference operator has kernel
\[\begin{aligned}
&(\widehat{ SB_{\sigma, U}S^{-1} }-\hat B_{\sigma, U})(\xi,\eta)\\
=&\big((1+|\xi|^2)^{\frac 12}-(1+|\eta|^2)^{\frac 12}\big) \begin{pmatrix}
\hat U(\xi-\eta)\cdot(i\eta) & \theta \hat \sigma(\xi-\eta)(i\eta^T)\\
\theta\hat\sigma(\xi-\eta)(i\eta) & \hat U(\xi-\eta)\cdot(i\eta)
\end{pmatrix}(1+|\eta|^2)^{-\frac 12}.
\end{aligned}\]
Note that $\frac{|\eta|}{\sqrt{1+|\eta|^2}}\leq 1$ and
\[\begin{aligned}
\big|(1+|\xi|^2)^{\frac 12}-(1+|\eta|^2)^{\frac 12}\big|
=&\big|\int_0^1\frac{d}{dp}(1+|p\xi+(1-p)\eta|^2)^\frac 12\ dp\big|\\
\leq&|\xi-\eta|\int_0^1\frac{|p\xi+(1-p)\eta|}{(1+|p\xi+(1-p)\eta|^2)^\frac 12}\ dp\leq|\xi-\eta|.
\end{aligned}\]
Therefore, we have
\begin{equation}\label{eq9.0}
\begin{aligned}|\widehat{SB_{\sigma, U}S^{-1}}-\hat B_{\sigma, U}|&\leq C\sum_\eta|\xi-\eta|(|\hat U(\xi-\eta)|+|\hat\sigma(\xi-\eta)|)\\
&\leq C\|(\sigma,U)\|_{s+1},\end{aligned}\end{equation}
which completes the proof.
\end{proof}
Consequently, for any $y\in X_{s+1}$, let $x=Sy$ and we have
\[\|M_{\sigma,U}(t,\tau)y\|_{s+1}=\|S^{-1}SM_{\sigma,U}S^{-1}x\|_{s+1}\leq |SM_{\sigma,U}(t,\tau)S^{-1}|_{L(X_s,X_s)}\|y\|_{s+1}.\]
Therefore, if $\|(\sigma,U)\|_{C^0([0,T],X_{s+1})}<\eps$, by \eqref{eq9.0}, we have
\begin{equation}\label{eq9.1}
|M_{\sigma,U}(t,\tau)|_{L(X_{s+1},X_{s+1})}\leq e^{(2C\eps+1)(t-\tau)},
\end{equation} 
where $C$ depends only on $s$ and $n$.

Define the following space
\[Z_T\triangleq \{(\sigma,U)\in C([0,T],X_s)\Big| (\sigma(0),U(0))=(\sigma_0,U_0), \|(\sigma,U)\|_{C^0([0,T],X_{s+1})}\leq\eps\},\]
equipped with the norm $\|(\sigma,U)\|_Z\triangleq \|(\sigma,U)\|_{C^0([0,T],X_s)}$. For any $(\sigma_0,U_0)\in X_{s+1}$ and $(\sigma,U)\in Z$, we let 
\[\mathcal F(\sigma,U,\sigma_0,U_0)(t)=M_{\sigma,U}(t,0)(\sigma_0,U_0).\] 
\begin{prop}\label{prop3}
Let $s>\frac n2$ and $\eps>0$ such that \eqref{eq7.1} holds. Then for any $\|(\sigma_0,U_0)\|_{s+1}\leq{\frac{\eps}{2e}}$, the mapping $\mathcal F(\cdot,\cdot,\rho_0,U_0)$ has a unique fixed point in $Z_T$ for some $T>0$. Consequently, for small initial data, \eqref{eq2} has a unique solution 
\begin{equation}\label{eq10}
(\sigma, U)\in C([0,T],X_{s+1})\cap C^1([0,T],X_s).
\end{equation}
\end{prop}
\begin{proof}
Let $T=\frac{1}{2C\eps+1}<1$, where $C$ depends on $s$ and $n$ appearing in Proposition \ref{prop2}. For $\|(\rho,U)\|_{C^0([0,T],X_{s+1})}\leq\eps$, by \eqref{eq7.1} and \eqref{eq9.1}, we have
\begin{equation}\label{map}
\begin{aligned}\|\mathcal F(\sigma,U,\sigma_0,U_0)\|_Z\leq &\|\mathcal F(\sigma,U,\sigma_0,U_0)\|_{C^0([0,T],X_{s+1})}\\
\leq &e^{(2C\eps+1)T}\frac{\eps}{2e}<\eps,\end{aligned}
\end{equation}
which means $\mathcal F(\cdot,\cdot,\sigma_0,U_0)$ maps $B_{\eps}(C^0([0,T],X_{s+1}))$ and $B_{\eps}(Z)$ into themselves, respectively. 

For any $(\sigma_1,U_1)$ and $(\sigma_2,U_2)$ in $Z$, we have
\[\begin{aligned}
&\|\mathcal F(\sigma_1,U_1,\sigma_0,U_0)-\mathcal F(\sigma_2,U_2,\sigma_0,U_0)\|_Z\\
=&\sup_{t\in[0,T]}\|\int_0^tM_{\sigma_1,U_1}(t,\tau)(B_{\sigma_1,U_1}(\tau)-B_{\sigma_2,U_2}(\tau))M_{\sigma_2,U_2}(\tau,0)(\rho_0,U_0)\ d\tau\|_s\\
\leq&\sup_{t\in[0,T]}\int_0^t e^{(c\eps+1)(t-\tau)}\|(\sigma_1-\sigma_2,U_1-U_2)(\tau)\|_s e^{(2c\eps+1)\tau}\|(\sigma_0,U_0)\|_{s+1}\ d\tau\\
\leq&e^{(2C\eps+1)T}\|(\sigma_0,U_0)\|_{s+1}\|(\sigma_1-\sigma_2,U_1-U_2)\|_Z\leq\frac 12\|(\sigma_1-\sigma_2,U_1-U_2)\|_Z,
\end{aligned}\]
which implies $\mathcal F$ is a contraction on $B_{\eps}(Z)$. Therefore, $\mathcal F$ has a unique fixed point in $Z$. To prove the fixed point is in $X_{s+1}$, we note that all iterations $(\sigma_n,U_n)\triangleq\mathcal F(\sigma_{n-1},U_{n-1},\sigma_0,U_0)$ are bounded in $X_{s+1}$ due to \eqref{map}. Consequently, there exists a subsequence $n_j$ such that $(\sigma_{n_j},U_{n_j})$ converges weakly in $X_{s+1}$ and thus strongly in $X_s$. The limit of such subsequence is exactly the fixed point of $\mathcal F$. Therefore, the local solution of \eqref{eq2} is also in $X_{s+1}$ and \eqref{eq10} holds for $T=\frac{1}{2C\eps+1}$ and all $\|(\rho_0,U_0)\|_{s+1}\leq\frac{\eps}{2e}$.
\end{proof}

\subsection{{\bf Asymptotic Behavior}.}\label{AB} In this subsection, we introduce a new method to study the asymptotic behavior of local solutions obtained in the previous subsection. The method is motivated by the following observation. From Lemma \ref{le1}, we see that the phase space $X_s$ can be decomposed into two subspaces $\tilde X_s$ and its orthogonal complement generated by the vector $(\sigma,U)=(1,0)$, which are invariant under the semigroup $S(t)$. Moreover, the semigroup has exponential decay on $\tilde X_s$. Intuitively, for small $(\sigma,U)$, the linear operator $B_{\sigma,U}$ is a small perturbation relative to $A$ and thus we expect a similar dynamical picture for $M_{\sigma,U}$. It turns out given any local solution $(\sigma(t),U(t))$, we can define a codimension one linear subspace $E_2(t)$ for each $t$, which is isomorphic  and close to $\tilde X_s$, such that
\begin{equation}\label{eq10.0}
X_s=E_1\oplus E_2(t),
\end{equation}
for $0\leq \tau\leq t\leq T$. Here $E_1=\{(\sigma,U)=(c,0),c\in \R\}$. By conjugating the flow from $E_2(\tau)$ onto $\tilde X_s$ and showing that the solution projected along $E_1$ direction can be slaved by the projection along $E_2(t)$ direction , we discover the exponential decay property of the solution. 

We begin with the construction of $E_2(t)$ in the following Lemma. We will use $Pf$ to denote the average of $f$ over $[0,1]^n$, i.e., $Pf=\int_{[0,1]^n} f\ dx$.
\begin{lemma}\label{le3.0}
Let $s>\frac n2$. There exists $\delta>0$ such that for all $\|(\sigma,U)\|_s<\delta$, there exist operators $(L_1,L_2): B_{\delta}(X_s)\longrightarrow L(X_s,\R)$ such that for all $(\sigma_1,U_1)\in \hat X_{s+1}$,
\begin{equation}\label{eq12}
(A+B_{\sigma,U})\begin{pmatrix}
L_1\sigma_1+L_2U_1+\sigma_1\\ U_1
\end{pmatrix}=\begin{pmatrix}
L_1\tilde \sigma_1+L_2\tilde U_1+\tilde\sigma_1\\
\tilde U_1
\end{pmatrix},
\end{equation}
where
\[\begin{aligned}
&\tilde \sigma_1=(I-P)(U\cdot\nabla(L_1\sigma_1+L_2 U_1+\sigma_1)+(1+\theta\sigma)\nabla\cdot U_1),\\
&\tilde U_1=(1+\theta\sigma)\nabla(L_1\sigma_1+L_2 U_1+\sigma_1)+U_1+U\cdot\nabla U_1.
\end{aligned}\]
\end{lemma}
\begin{proof}
We note \eqref{eq12} is equivalent to 
\begin{equation}\label{eq13}
\begin{aligned}
0=&\mathcal G(L_1,L_2;\sigma,U)(\sigma_1,U_1)\\
\triangleq &P(U\cdot\nabla(L_1\sigma_1+L_2 U_1+\sigma_1)+(1+\theta\sigma)\nabla\cdot U_1)\\
&-L_1\big((I-P)U\cdot\nabla(L_1\sigma_1+L_2 U_1+\sigma_1)+(1+\theta\sigma)\nabla\cdot U_1\big)\\
&-L_2\big((1+\theta\sigma)\nabla(L_1\sigma_1+L_2 U_1+\sigma_1)+U_1+U\cdot\nabla U_1\big)\\
= &P(U\cdot\nabla\sigma_1+\theta\sigma\nabla\cdot U_1)-L_1\big((I-P)U\cdot\nabla\sigma_1\\&\hspace{1cm}+(1+\theta\sigma)\nabla\cdot U_1\big)
-L_2\big((1+\theta\sigma)\nabla\sigma_1+U_1+U\cdot\nabla U_1\big).
\end{aligned}
\end{equation} 
Consider $\mathcal G$ as a mapping from $\tilde X_{-s}\times X_{s}$ to $\tilde X_{-(s+1)}$. It is obvious that $\mathcal G(0,0;0,0)=0$ and
\[D_{L_1,L_2}\mathcal G(0,0;0,0)(\tilde L_1,\tilde L_2)=(\tilde L_1,\tilde L_2)\begin{pmatrix}
0 & -\nabla\cdot\\
-\nabla & -I
\end{pmatrix}.\]
Since $\begin{pmatrix}
0 & -\nabla\cdot\\
-\nabla & -I
\end{pmatrix}$ is an isomorphism between $\tilde X_{r}$ and $\tilde X_{r-1}$ for any $r$, the implicit function theorem implies there exists $\delta>0$ such that for any $\|(\rho,U)\|_{s}<\delta$, there exists a unique pair of operators $(L_1,L_2)\in L(\hat X_s,\R)$ such that \eqref{eq12} (or equivalently \eqref{eq13}) holds. Moreover, $(L_1,L_2)$ are smooth in $(\sigma,U)$ and 
\begin{equation}\label{eq13.1}
L_1(0,0)=0\ , \ L_2(0,0)=0\ , \ |D(L_1,L_2)|_{C^0(B_\delta(\tilde X_s),L(\tilde X_s,\tilde X_{-s}))}\leq C_1\delta,
\end{equation}
where $C_1$ is independent of $(\sigma, U)$.
\end{proof}
We extend $L_1$ to a linear functional on $H^s$ (still denoted by $L_1$) by setting 
\begin{equation}\label{eq12.1}
L_1f=L_1(I-P)f\ , \ f\in H^s.\end{equation}
Consequently, we have 
\begin{equation}\label{eq14}
L_1^2=0\ , \ L_1L_2=0\ , \ (I+L_1)^{-1}=I-L_1,
\end{equation}
and
\begin{equation}\label{eq15}
\begin{pmatrix}
I+L_1 & L_2\\
0 & I
\end{pmatrix}^{-1}=\begin{pmatrix}
I-L_1 & -L_2\\ 0 & I
\end{pmatrix}.
\end{equation}
Let $(\sigma(t),U(t))$ be a local solution of \eqref{eq2} and we define
\[E_2(t)=\{\begin{pmatrix}
I+L_1(\sigma(t),U(t)) & L_2(\sigma(t),U(t))\\
0 & I
\end{pmatrix}\begin{pmatrix}
\sigma_1 \\ U_1
\end{pmatrix}\Big|(\sigma_1, U_1)\in \tilde X_{s}\}.\]
Recall that $E_1=\{(\sigma,U)=(c,0),c\in \R\}$. It is clear that
\[X_s=E_1\oplus E_2(t).\]
We then decompose $(\sigma(t),U(t))$ dynamically as
\begin{equation}\label{eq15.1}
\begin{pmatrix}
\sigma(t)\\ U(t)
\end{pmatrix}=\begin{pmatrix}
c(t)\\0
\end{pmatrix}+\begin{pmatrix}
I+L_1 & L_2\\ 0& I
\end{pmatrix}\begin{pmatrix}
(I-P)\sigma(t)\\ U(t)
\end{pmatrix}.
\end{equation}
Plugging such decomposition into \eqref{eq2}, we have
\[
\begin{aligned}
&\begin{pmatrix}
c\\0
\end{pmatrix}_t+\begin{pmatrix}
I+L_1 & L_2\\ 0 & I
\end{pmatrix}_t\begin{pmatrix}
(I-P)\sigma\\ U
\end{pmatrix}+\begin{pmatrix}I+L_1 & L_2\\ 0 & I\end{pmatrix}\begin{pmatrix}(I-P)\sigma\\ U\end{pmatrix}_t\\
=&(A+B_{\sigma,U})\begin{pmatrix}
I+L_1 & L_2\\ 0 & I
\end{pmatrix}\begin{pmatrix}
(I-P)\sigma\\ U
\end{pmatrix}.
\end{aligned}
\]
We note that 
\[\begin{pmatrix}
I+L_1 & L_2\\ 0 & I
\end{pmatrix}_t\begin{pmatrix}
(I-P)\sigma\\ U
\end{pmatrix}=\begin{pmatrix}
\partial_tL_1 & \partial_tL_2\\ 0 & 0
\end{pmatrix}\begin{pmatrix}
(I-P)\sigma\\ U
\end{pmatrix}\in E_1.\]
By the fact $E_1\cap E_2(t)=\{0\}$ and \eqref{eq12}, we obtain
\begin{equation}\label{eq15.2}
c_t(t)=-\partial_t L_1 (I-P)\sigma-\partial_t L_2 U,
\end{equation}
and 
\begin{equation}\label{eq15.3}
\begin{aligned}
&\begin{pmatrix}
I+L_1 & L_2\\ 0& I
\end{pmatrix}\begin{pmatrix}
(I-P)\sigma\\ U
\end{pmatrix}_t\\
=&(A+B_{\sigma,U})\begin{pmatrix}
I+L_1 & L_2\\ 0& I
\end{pmatrix}\begin{pmatrix}
(I-P)\sigma\\ U
\end{pmatrix}.
\end{aligned}
\end{equation}
Next we conjugate flows from $E_2(t)$ onto $\tilde X_s$. Let
\[\begin{pmatrix}
\sigma_1 \\ U_1
\end{pmatrix}=\begin{pmatrix}
(I-P)\sigma\\ U
\end{pmatrix}=\begin{pmatrix}
I-L_1 & -L_2\\
0 & I
\end{pmatrix}\begin{pmatrix}
\sigma -c\\ U
\end{pmatrix}\in \tilde X_{s+1}.\]
A straightforward computation shows \eqref{eq15.3} is equivalent to 
\begin{equation}\label{decay eq}
\begin{pmatrix}
\sigma_1\\ U_1
\end{pmatrix}_t=(A+B_{c+(I+L_1)\sigma_1+L_2 U_1,U_1}+\tilde B_{\sigma_1,U_1})\begin{pmatrix}
\sigma_1\\ U_1
\end{pmatrix},
\end{equation}
where $\tilde B_{\sigma_1,U_1}=(\tilde B_{i,j})$ for $i,j=1,2$ and 
\[\begin{aligned}
\tilde B_{1,1}=&L_1U_1\cdot\nabla+L_2(1+\theta(c+(I+L_1)\sigma_1+L_2 U_1))\nabla,\\
\tilde B_{1,2}=&L_1\theta(c+(I+L_1)\sigma_1+L_2 U_1)\nabla\cdot+L_1\nabla\cdot+L_2(I+U_1\cdot\nabla),\\
\tilde B_{2,1}=&0\ ,\ 
\tilde B_{2,2}=0.
\end{aligned}\]
Since we already have a local solution $(\sigma,U)$, the linear operator on the right hand side of \eqref{decay eq} generates an evolutionary system $\tilde M_{\sigma_1,U_1}(t,\tau)$ on $\tilde X_s$, which also maps $\tilde X_{s+1}$ to $\tilde X_{s+1}$.

\begin{lemma}\label{le3}
Under the hypothesis in Proposition \ref{prop2}, there exists $\tilde K>0$ such that the evolutionary system $\tilde M(t,\tau)$ satisfies
\begin{equation}\label{eq11.1}
|\tilde M(t,\tau)|_{L(\hat X_{s+1},\hat X_{s+1})}\leq \tilde Ke^{-\frac 14(t-\tau)}\ \ \text{for}\ \ 0\leq\tau\leq t\leq T.
\end{equation}
\end{lemma}
\begin{proof}
%It is clear that the operator $A+B_{\sigma,U}$ has a simple eigenvalue at $0$ with eigenspace $\{c\in\R|(\sigma,U)=(c,0)\}$. Therefore, $E_1(t)=\R$. To identify $E_2(t)$, we look for linear operators $(L_1,L_2):\tilde X_s \to \R$, which depend on $(\sigma, U)$, such that for any $(\sigma_1,U_1)\in\tilde X_{s+1}$,The special form on the right hand side of \eqref{eq12} implies
From the second estimate in \eqref{eq3.1}, we can deduce for small $\eps$ and any $\lambda>-\frac 14$,
\[\begin{aligned}
&\|(\lambda-(A+B_{c+(I+L_1)\sigma_1+L_2 U_1,U_1}+\tilde B_{\sigma_1,U_1}))^{-1}\|_{L(\tilde X_{s},\tilde X_{s})}\\
\leq &\|(I+(B_{c+(I+L_1)\sigma_1+L_2 U_1,U_1}+\tilde B_{\sigma_1,U_1})(\lambda-A)^{-1})^{-1}\|\|(\lambda-A)^{-1}\|_{L(\tilde X_{s},\tilde X_{s})}\\
\leq & 2\frac{K}{\lambda+\frac 12}\leq 4K.
\end{aligned}\]
By Theorem V.1.11 in \cite{EN00}, there exists $\tilde K>0$ such that \eqref{eq11.1} holds. The proof is completed.
\end{proof}
The above lemma implies $(\sigma_1,U_1)$ decays as long as $(c,\sigma_1,U_1)$ are small. Thus, if we can bound $c$ in terms of $(\sigma_1,U_1)$, the local solution can be extended to a global one, which also satisfies the decay estimate. To confirm this fact, we first note that if $\rho$ satisfies \eqref{eq1}, then $\rho-1$ satisfies the Poincare\texttt{\char13}s inequality, namely,
\[\|\rho-1\|_{L^2}\leq C\|\nabla\rho\|_{L^2}.\]
According to the definition of $\sigma=\frac{\rho^\theta-1}{\theta}$, we also have 
\begin{equation}\label{Poincare}
\|\sigma\|_{L^2}\leq C\|\nabla\sigma\|_{L^2}.
\end{equation}
From \eqref{eq15.2}, we can write the solution $c(t)=c(\sigma_1,U_1)(t)$, where we recall $(\sigma_1,U_1)=((I-P)\sigma,U)$. By the standard ODE theory, $c(\cdot,\cdot)$ is smooth in $(\sigma_1,U_1)$. We claim that $c(0,0)=0$. If it is not true, by \eqref{eq15.2}, 
\[c_t=0\Longrightarrow c(t)=c,\]
which contradicts to \eqref{Poincare}. Consequently, there exists $C$ independent of $(\sigma_1,U_1)$ such that
\[|c(\sigma_1,U_1)(t)|\leq C(\|\sigma_1(t)\|_{s+1}+\|U_1(t)\|_{s+1}).\]
Using such estimate and Lemma \ref{le3}, we conclude $\|\sigma_1\|_{s+1}+\|U_1\|_{s+1}$ decays exponentially. Therefore, the $H^{s+1}$ norm of $(\sigma,U)$ also decay exponentially. The proof of the Main Theorem is completed.

\end{document}